\documentclass[12pt]{amsart}
\usepackage{amsfonts}
\usepackage{amssymb}
\usepackage{mathtools}
\usepackage[all]{xy}
\usepackage{amsmath, amssymb, amsfonts, amstext, amsthm}
\usepackage[mathscr]{euscript}
\usepackage{enumerate}
\usepackage{verbatim}
\usepackage{hyperref}
\def \ra{\rightarrow}
\def \mfk{\mathfrak}
\def \mcO{\mathcal{O}}

\def \glH{H^0(C,A)}

\def \tildeX{\widetilde{X}}
\def \mbb{\mathbb}

\def \parK{A\otimes\mcO_X(-C)|_C}
\def \5tildeF{\widetilde{F}}
\def \tildeC{\widetilde{C}}
\def \pushC'{i'_*f_*}

\def \Ln{L_n}
\def \xhra{\xhookrightarrow}

\def \tildeG{\widetilde{G}}
\newcommand{\trm}{\textrm}
\newcommand{\lr}{\longrightarrow}
\newcommand{\hra}{\hookrightarrow}

\newtheorem{inducedstructureonG}{Lemma}[section]

\newtheorem{pbstable}[inducedstructureonG]{Remark}

\title{Semistability of Lazarsfeld-Mukai bundles via parabolic structures}
\author[P. Narayanan]{Poornapushkala Narayanan}
\address{Department of Mathematics, Indian Institute of Technology Madras, Chennai - 600036.}
\email{poorna.p.narayanan@gmail.com}
\theoremstyle{definition}
\thanks{Mathematics Classification numbers: 14C20, 14E20, 14J60}

\newtheorem{quasiparabolic}{Definition}[section]

\newtheorem{morphismofparsheaf}[quasiparabolic]{Definition}
\newtheorem{inducedparabolicintro}[quasiparabolic]{Definition}

\newtheorem{parabolicweight}[quasiparabolic]{Definition}
\newtheorem{parabolicslope}[quasiparabolic]{Definition}
\newtheorem{parabolicstability}[quasiparabolic]{Definition}

\begin{document}
\begin{abstract}
  Our aim in this article is to produce new examples of semistable
  Lazarsfeld-Mukai bundles on smooth projective surfaces $X$ using the
  notion of parabolic vector bundles. In particular, we associate
  natural parabolic structures to any rank two (dual) Lazarsfeld-Mukai
  bundle and study the parabolic stability of these parabolic bundles.
  We also show that the orbifold bundles on Kawamata coverings of $X$
  corresponding to the above parabolic bundles are themselves certain
  (dual) Lazarsfeld-Mukai bundles. This gives semistable
  Lazarsfeld-Mukai bundles on Kawamata covers of the projective plane
  and of certain K3 surfaces.
\end{abstract}

\maketitle
\section{Introduction}
In recent years, the study of certain vector bundles called
Lazarsfeld-Mukai bundles (LM bundles) has gained
prominence. Investigating the properties of these naturally occurring
bundles is of increasing interest.  This can be attributed to the deep
applications of these bundles in Brill-Noether theory of curves,
especially those lying on K3 surfaces.

LM bundles were first used by Lazarsfeld \cite{RL} to prove Petri's
conjecture and by Mukai \cite{Mu} in the classification of certain
Fano manifolds. They have also been useful in studying the constancy
of gonality, Clifford index and Clifford dimension of smooth
projective curves belonging to ample or globally generated linear
systems on K3 surfaces \cite{CP, GL, Kn}.  Voisin's proof of the
generic Green's conjecture employed these bundles \cite{CV1,CV2}, and
Aprodu and Farkas \cite{AF} use LM bundles and their parameter spaces
while proving Green's conjecture for curves on a K3 surface.  The
(semi)stability properties of LM bundles over K3 surfaces were studied
by Lelli-Chiesa \cite{ML}, and of certain LM bundles over Jacobian
surfaces and higher dimensional varieties by us \cite{NP1,NP2}.

In this article, our aim is to produce new examples of semistable
Lazarsfeld-Mukai bundles via the general theory of parabolic vector
bundles. In particular, we associate certain parabolic vector bundles
to a rank two LM bundle and study the related notions of parabolic
stability.  We refer to $\mathcal{x}\,$\ref{parabolic-preliminaries}
for some preliminaries on parabolic vector bundles.

Suppose $X$ is a smooth projective surface over $\mbb{C}$ and
$C\xhra{i} X$ is a smooth curve. Consider a globally generated line
bundle $A$ on $C$, with $\trm{dim}\,H^0(C,A)=2$. Then the rank two LM
bundle associated to the pair $(C,A)$ is the dual of the vector bundle
$F$, where $F$ is defined by the following short exact sequence:
\begin{equation*}\label{introLMseq}
0\ra F\ra \glH\otimes \mcO_X\ra i_*A\ra 0\,. 
\end{equation*}
Note that $F$ comes in-built with a parabolic structure given by the
following filtration with associated weights $0\leq a_1<a_2<1$:
\begin{equation}
  \boxed{\mathfrak{F}_C: F|_C\supset_{a_1} \parK\supset_{a_2} 0\,.}\nonumber
\end{equation}
We thus obtain a parabolic vector bundle
$F_*=(F,\mathfrak{F}_C,a_1,a_2)$ which we call the \emph{(dual)
  parabolic LM bundle}.

Our second parabolic bundle arises as follows. Consider a point $q$ of
multiplicity one on $C$. Consider the blow up of $X$ at $q$, say
$\tildeX$. Let $\pi:\tildeX\ra X$ denote the blow down map and
$E\subset\tildeX$, the exceptional divisor. The pulled back rank two
vector bundle $\5tildeF:=\pi^*F$ on $\tildeX$ admits a parabolic
structure along the exceptional divisor $E$ by associating weights
$0\leq b_1<b_2<1$:
\begin{equation}
\boxed{\widetilde{\mathfrak{F}}_E:\5tildeF|_E=F(q)\otimes\mcO_E\supset_{b_1} M\supset_{b_2} 0\,.}\nonumber
\end{equation}  
Here, $M=\trm{kernel}(\5tildeF|_E\rightarrow H^0(C,A)\otimes \mcO_E)$.
We call $\5tildeF_*:=(\5tildeF,\widetilde{\mathfrak{F}}_E,b_1,b_2)$
the \emph{(dual) blown up parabolic LM bundle}. Refer $\mathcal{x}\,$
\ref{Parabolic structure on LM bundles} for details.

In $\mathcal{x}\,$\ref{stability of parabolic bundles}, we study the
parabolic stability of the two parabolic vector bundles mentioned
above. Let $L$ be an ample line bundle on $X$. Then we have the
following theorem.  \theoremstyle{plain}
\newtheorem{thm}{Theorem}[section]
\begin{thm}\label{introparstable}
  Suppose that the vector bundle $F$ is $\mu_L$-stable. Then the
  (dual) parabolic LM bundle $F_*$ on $X$ is parabolic $\mu_L$-stable
  if the weights $a_1$, $a_2$ are such that:
 \begin{enumerate}
 \item[(a)] $a_2-a_1 < \frac{2}{(C\cdot L)}$, if the intersection
   number $(C\cdot L)$ is even, or
 \item[(b)] $a_2-a_1< \frac{1}{(C\cdot L)}$, if $(C\cdot L)$ is odd.
 \end{enumerate}  
\end{thm}
Next, consider the line bundle
$\Ln := n \pi^*L \otimes \mcO_{\tildeX}(-E)$ on $\tildeX$ which is
ample for $n$ sufficiently large. In this case, we have the following
theorem.
\begin{thm}\label{introparstable2}
  Suppose that $F$ is $\mu_L$-stable. Then, there is an integer $n_0$
  such that for all $n\geq n_0$, $\5tildeF_*$ is parabolic
  $\mu_{\Ln}$-stable for all weights $b_1$ and $b_2$.
 \end{thm}
 In \cite{IB}, Biswas established a two way relationship between
 parabolic vector bundles and orbifold vector bundles. Motivated by
 this, in $\mathcal{x}\,$\ref{Orbifold LM bundle}, we obtain the
 orbifold bundle on Kawamata covers of $X$ associated to the (dual)
 parabolic LM bundle $F_*$ for weights of the form $a_1=0$ and
 $a_2=\frac{N-m}{N}$, where $1\leq m<N$ are positive integers. Let
 $p:Y\ra X$ be the Kawamata covering of $X$ such that $p^*C$ is a
 non-reduced divisor $NC'$ with $C':= (p^*C)_{\text{red}}$ a smooth
 curve on $Y$.  Then we have the following theorem.
\begin{thm}\label{intro-parabolic-orbifold}
  The orbifold vector bundle on $Y$ corresponding to the (dual)
  parabolic LM bundle $(F,\mfk{F}_C,0,\frac{N-m}{N})$ on $X$ is the
  dual LM bundle on $Y$, say $F'$, associated to the triple
  $(mC',A', H^0(C,A))$ given by the short exact sequence
$$0\ra F'\ra H^0(C,A)\otimes \mcO_Y\ra j'_*A'\ra 0\,.$$
Here $A'$ is the pullback of the line bundle $A$ 
from $C$ to the curve $mC'\xhra{j'}Y$ (the curve $mC'$ is non-reduced when $m>1$).
\end{thm}
It is interesting to observe that the orbifold bundles are themselves
certain dual LM bundles on the Kawamata covers. This further
highlights that LM bundles are vector bundles with desirable
properties.

As a consequence of the above correspondence, we obtain new examples
of semistable LM bundles on the Kawamata covers of $\mbb{P}^2$ and K3
surfaces in $\mathcal{x}\,$\ref{semist-appl-section}.
\begin{thm}\label{intronewsemistLMbundles}
  Suppose $X$ is a smooth projective surface and $C$ is a smooth curve
  on $X$.  Let $A$ be an ample line bundle on $C$ and
  $V\subset H^0(C,A)$ be a general two dimensional subspace with the
  linear series $\mbb{P}V$ base-point free. For any positive integer
  N, consider the Kawamata covering $p:Y\ra X$ of $X$ such that $p^*C$
  is a non-reduced divisor $NC'$ with $C':=(p^*C)_{\emph{red}}$ a
  smooth curve on $Y$. Let $1\leq m < N$ be an integer and $A'$ denote
  the pullback of $A$ from $C$ to (the possibly non-reduced curve)
  $mC'$.
\begin{enumerate}
\item[(a)] Consider $X=\mbb{P}^2$ and $C\in |\mcO(d)|$, where $d$ is
  odd. Set $A$ to be of the form $\mcO(rd)|_C$ for $r\geq 1$.  Choose
  $m$ such that $\frac{N-m}{N}<\frac{1}{d}$. Then the LM bundle on $Y$
  associated to the triple $(mC',A',V)$ is
  $\mu_{p^*\mcO(1)}$-semistable.
\item[(b)] Let $X$ be a smooth projective K3 surface and $L$ be an
  ample line bundle on $X$ such that a general curve $C\in |L|$ has
  genus $g$, Clifford dimension 1 and maximal gonality. Let $A$ be a
  complete base-point free $g^1_d$ on $C$ where the Brill Noether
  number $\rho(g,1,d)>0$. Set $V=H^0(C,A)$. Let $m$ be such that
  $\frac{N-m}{N}<\frac{1}{g-1}$. Then the LM bundle on $Y$ associated
  to the triple $(mC',A',H^0(C,A))$ is $\mu_{p^*L}$-semistable.
 \end{enumerate}
 We have the following as particular cases of (a) and (b)
 respectively. Let $l=\trm{deg}\, A'$.
 \begin{enumerate}
 \item[(a')] Let $d=1$, i.e. $C\in|\mcO(1)|$ is a line and set $m=1$.
   Then there is an irreducible component of the Brill-Noether variety
   $\mathcal{G}^1_l(|\mcO_Y(C')|)$ corresponding to
   $\mu_{p^*\mcO(1)}$-semistable LM bundles.
  \item[(b')] Suppose $X$ be a smooth projective K3 surface and $L$ is as above
 which in addition satisfies $(L^2)=2$. 
 Let $m=1$. Then there is an irreducible component
 of the Brill-Noether variety $\mathcal{G}^1_l(|\mcO_Y(C')|)$ corresponding
 to $\mu_{p^*L}$-semistable LM bundles. 
\end{enumerate}
\end{thm}
\subsection*{Acknowledgements}
I thank Dr. Jaya NN Iyer for introducing me to this problem and for
guiding me during the course of this project. I also thank
Dr. T. E. Venkata Balaji and Prof. D. S. Nagaraj for helpful
discussions.
\section{Preliminaries}\label{parabolic-preliminaries}
\subsection{Parabolic sheaves}
Parabolic bundles were first introduced over curves by Seshadri
\cite{Se} and Mehta \cite{MS}. This was generalized to higher
dimensions by Bhosle \cite{UB} and by Maruyama-Yokogawa \cite{MY}.
Consider a connected smooth projective variety $X$ and an effective
divisor $D$ on $X$.
\begin{quasiparabolic}\label{qparabolicintro}
 Let $E$ be a coherent torsion-free $\mcO_X$-module. A \emph{quasi-parabolic structure} 
 on $E$ along $D$ is defined by giving a filtration of $E|_D$ of the form:
 $$\mfk{F}_D:E|_D=F_1(E|_D)\supset F_2(E|_D)\supset\cdots \supset F_l(E|_D)\supset F_{l+1}(E|_D)=0\,.$$
 A \emph{parabolic structure} on $E$ is a quasi-parabolic structure as
 above together with a system of \emph{parabolic weights}
 $a_1,a_2,\cdots,a_l$ such that $0\leq a_1<a_2<\cdots<a_l<1$.
\end{quasiparabolic}
We denote $E$ with the parabolic structure by $(E,\mfk{F}_D,a_*)$ or
by $E_*$.  If $E$ is locally free, then $E_*$ is called a
\emph{parabolic bundle}.

In the following definitions, let $E$ be a torsion-free coherent sheaf
with parabolic structure $(E,\mfk{F}_D,a_i(E))$ where $i=1,2,\ldots,l$
and
 \begin{equation*}\label{todefinebasicparabolic}
   \mfk{F}_D:E|_D=F_1({E}|_D)\supset F_2({E}|_D)\supset\cdots \supset F_l({E}|_D)\supset F_{l+1}({E}|_D)=0\,. 
 \end{equation*}
 \begin{morphismofparsheaf}\label{morphismofparsheaf}\cite[Definition
   1.10]{UB} Suppose that $G$ is another torsion-free coherent sheaf
   with quasi-parabolic structure along $D$ prescribed by the
   filtration $\mfk{F}'_D$ of ${G}|_D$:
 $$\mfk{F}'_D:G|_D=F_1(G|_D)\supset F_2(G|_D)\supset\cdots \supset F_r(G|_D)\supset F_{r+1}(G|_D)=0\,. $$
 together with weights $\{a_i(G)\}$. \emph{A morphism of parabolic
   sheaves} $G\ra E$ is a morphism of sheaves $g:G\ra E$ such that
 whenever $a_i(G)>a_j(E)$, then
 $g|_D(F_i(G|_D))\subset F_{j+1}(E|_D)$.
\end{morphismofparsheaf}
Consider now a subsheaf $F\hra E$ such that the quotient $Q=E/F$ is
torsion-free. In this case, $F|_D\ra E|_D$ is an injection
(cf. \cite[$\mathcal{x}\,$1.8]{UB}). Then we have the following
induced parabolic structure on the subsheaf $F$.
 \begin{inducedparabolicintro}\label{inducedparabolicintro} \cite[$\mathcal{x}\,$1.8]{UB}
   Let $e:F|_D\hra E|_D$ denote the inclusion. We define a
   quasi-parabolic structure on $F$ along $D$ by pulling back the flag
   $\mfk{F}_D$ by $e$.
 \begin{displaymath}
   \xymatrix{\mfk{F}_D:E|_D=F_1(E|_D) & F_2(E|_D)\ar@{_{(}->}[l]\cdots & F_l(E|_D) \ar@{_{(}->}[l]  & 0 \ar@{_{(}->}[l]\\
     e^{-1}(F_1(E|_D))=F|_D\ar@{^{(}->}[u] & e^{-1}(F_2(E|_D))\ar@{_{(}->}[l]\ar@{^{(}->}[u]\cdots & e^{-1}(F_l(E|_D))\ar@{_{(}->}[l]\ar@{^{(}->}[u] & 0\ \ar@{_{(}->}[l]}
 \end{displaymath}
 Define the flag $\mfk{F}'_D$ of $F|_D$ by choosing the irredundant
 terms of the filtration and call them $F_j(F|_D)$ where
 $j\in J\subset\{1,2,\ldots,l\}$. The weight $a_j(F)$ associated to
 $F_j(F|_D)$ is determined as follows. If $i$ is the largest integer
 such that $F_j(F|_D)=e^{-1}(F_i(E|_D))$, then define
 $a_j(F):=a_i(E)$. This parabolic structure on $F$ is called the
 \emph{induced parabolic structure}.
\end{inducedparabolicintro}
With this definition, the inclusion $F\hra E$ is in fact a morphism of
parabolic sheaves (as in Definition \ref{morphismofparsheaf}).  Hence
$F$ is \emph{parabolic subsheaf} of $E$.  The induced structure on $F$
is the ``maximum parabolic structure'' making the inclusion morphism a
morphism of parabolic sheaves.  We next define the concepts of
parabolic slope and parabolic stability. Let $L$ be an ample line
bundle on $X$ with $\text{dim}\, X = n$.
\begin{parabolicweight}\label{parabolicweight}
  The \emph{parabolic weight} of $E$ denoted by $\mu\trm{-wt}(E)$ is
  defined as:
 $$\mu\trm{-wt}(E)=\sum_{i=1}^l a_i(E)[r(F_i(E|_D))-r(F_{i+1}(E|_D))]D\cdot L^{n-1}.$$
 Here $r(F_i(E|_D))$ denotes the rank of $F_i(E|_D)$.
\end{parabolicweight}
The motivation for the definition is the following.  Let $i:D\hra X$
denote the inclusion. Then $c_1(i_*F_i(E|_D))=r(F_i(E|_D))D$,
cf. \cite[Lemma 2.2]{UB}.  Thus the above definition is the same as:
$$\mu\trm{-wt}(E)=\sum_{i=1}^l a_i(E)[c_1(i_*F_i(E|_D))\cdot L^{n-1}-c_1(i_*F_{i+1}(E|_D))\cdot L^{n-1}]\,.$$
\begin{parabolicslope}
  The \emph{parabolic slope} of $E$ with respect to $L$ is:
 $$\trm{par}\mu_L(E)=\frac{c_1(E)\cdot L^{n-1}+\mu\trm{-wt}(E)}{\trm{rank}\,E}\,.$$
\end{parabolicslope}
\begin{parabolicstability}\label{parabolicstability}
  Let $E$ be a parabolic sheaf as above. Then $E$ is said to be
  \emph{parabolic} $\mu_L$-\emph{semistable} (resp. \emph{parabolic}
  $\mu_L$-\emph{stable}) if for every subsheaf $F$ of $E$ with
  torsion-free quotient such that $0<\trm{rank}\,F<\trm{rank}\,E$
  equipped with the \emph{induced parabolic structure}, one has the
  inequality of slopes, $\trm{par}\mu_L(F)\leq \trm{par}\mu_L(E)$
  (resp.  $\trm{par}\mu_L(F)<\trm{par}\mu_L(E))\,.$
\end{parabolicstability}
\theoremstyle{definition}
\newtheorem{orbbundle}[quasiparabolic]{Definition}
\newtheorem{orbstable}[quasiparabolic]{Definition}
\theoremstyle{plain}
\newtheorem{orbstableisstable}[quasiparabolic]{Proposition}
\subsection{Orbifold sheaves} Let $Y$ be a connected smooth projective
variety of dimension $n$. Let $\trm{Aut}(Y)$ denote the group of
algebraic automorphisms of $Y$. Let $G$ be a finite group and
$\rho:G\ra \trm{Aut}(Y)$ be an injective group homomorphism.
\begin{orbbundle}
  An \emph{orbifold sheaf} on $Y$ is a coherent torsion-free sheaf $E$
  on $Y$ together with a lift of the action of $G$ to $E$. That is,
  $G$ acts on the total space of stalks of $E$ and the automorphism of
  the space of stalks for the action of any $g \in G$ is a coherent
  sheaf isomorphism between $E$ and $\rho(g^{-1})^*E$. If $E$ is
  locally free, then $E$ is called an \emph{orbifold bundle}.
\end{orbbundle}
A subsheaf $F$ of an orbifold sheaf $E$ with torsion-free quotient
$E/F$ is said to be $G$-saturated if $F$ is left invariant by the
action of $G$. Hence $F$ has an induced orbifold sheaf structure. Let
$\widetilde{L}\in\trm{Pic}(Y)$ be an ample line bundle which is also
an orbifold bundle.  Recall that the usual Mumford-Takemoto slope of
$E$ with respect to $\widetilde{L}$ is
$\mu_{\widetilde{L}}(E)=\frac{c_1(E)\cdot
  \widetilde{L}^{\trm{dim}\,Y-1}}{\text{rank}\,E}$.
\begin{orbstable}
  An orbifold sheaf $E$ is said to be orbifold semistable (resp.
  orbifold stable) if for all $G$-saturated subsheaves $F$ of $E$ such
  that $0<\trm{rank}\,F<\trm{rank}\,E$ one has
  $\mu_{\widetilde{L}}(F)\leq\mu_{\widetilde{L}}(E)$
  (resp. $\mu_{\widetilde{L}}(F)< \mu_{\widetilde{L}}(E)$).
\end{orbstable}
Note that if $E$ is an orbifold sheaf which is $\mu_{\widetilde{L}}$-
(semi)stable, then clearly $E$ is orbifold (semi)stable.  There is a
converse statement in case of semistability which is as follows.
\begin{orbstableisstable}\cite[Lemma 2.7]{IB}
 An orbifold bundle is orbifold semistable if and only if it is
 semistable in the sense of Mumford-Takemoto.
\end{orbstableisstable}
Let $p:Y\ra Y/G=:X$ be the projection. Assume that the quotient $X$ is
a smooth variety. Let $E$ be an orbifold sheaf on $Y$.  The direct
image sheaf $p_*E$ on $X$ naturally has a $G$-action compatible with
the trivial action of $G$ on the quotient $X$.  Let $(p_*E)^G$ denote
the subsheaf of $G$-invariants in $p_*E$.  \theoremstyle{definition}
\newtheorem{invariantsubsheaf}[quasiparabolic]{Definition}
\begin{invariantsubsheaf}
  Let $Av$ denote the \emph{averaging Reynolds operator}
  $Av:p_*E\ra p_*E$ given by $Av=\frac{1}{|G|}\sum_{h\in G} t_h$. Here
  $t_h$ is the endomorphism of $p_*E$ induced by the action of
  $h\in G$. The image of $Av$ is the invariant subsheaf
  $(p_*E)^G$. Thus we have a short exact sequence:
$$0\ra (p_*E)^G\ra p_*E\ra  \frac{p_*E}{(p_*E)^G}\ra 0\,.$$
\end{invariantsubsheaf}
When $E$ is an orbifold bundle, the direct image sheaf $p_*E$ is also
locally free. In this case, since the above exact sequence splits, the
subsheaf of invariants $(p_*E)^G$ is locally free as well.
\section{Parabolic structures associated to LM
  bundles}\label{Parabolic structure on LM bundles}
Consider a smooth projective surface $X$ over the field $\mbb{C}$. For
any smooth curve $C\xhra{i} X$, consider a complete base-point free
$g^1_d$ on $C$, say $A$ (i.e. $A\in \trm{Pic}^d(C)$ is base-point free
with $h^0(A)=2$). We then have the following short exact sequence
corresponding to the linear series $(A,H^0(C,A))$ on $C$:
\begin{equation}\label{evseqonC}
 0\lr A^{\vee}\lr \glH\otimes\mcO_C\lr A\lr 0\,.
\end{equation}
The dual LM bundle $F_{C,A}=:F$ associated to $(C,A)$ on $X$ is 
given by the short exact sequence:
 \begin{equation}\label{LMseq}
 0\lr F\lr \glH\otimes\mcO_X\lr i_*A\lr 0\,. 
 \end{equation}
 Since $F$ is a rank two vector bundle with determinant $\mcO_X(-C)$,
 we have
 $F^{\vee}\simeq F\otimes (\trm{det}\,F)^{\vee} \simeq F\otimes
 \mcO_X(C)$. Note that $F^{\vee}$ is the LM bundle associated to
 $(C,A)$ and sits in the following sequence:
 \begin{equation}\label{dualLMseq}
 0\lr \glH^{\vee}\otimes\mcO_X\lr F\otimes\mcO_X(C)\lr i_*(A^{\vee}\otimes\mcO_X(C)|_C)\lr 0\,.
\end{equation}
We now associate parabolic vector bundles to the dual LM bundle $F$ in
two ways.
\subsection{Parabolic LM bundle}\label{structure1}
We first associate a parabolic structure to the vector bundle $F$
along the divisor $C$ on $X$. Restrict \eqref{LMseq} to the curve $C$
to get the following exact sequence on $C$:
\begin{equation*}
0\lr A\,\otimes\,\mcO_X(-C)|_C\lr F|_C\lr \glH\otimes\mcO_C\lr A\lr 0\,.
\end{equation*}
From the above sequence and \eqref{evseqonC}, we also get
\begin{equation}\label{parabonF}
0\lr A\,\otimes\,\mcO_X(-C)|_C\lr F|_C\lr A^{\vee}\lr 0\,.
\end{equation}
This gives a quasi-parabolic structure on $F$ along the curve $C$, the
flag being given by:
\begin{equation*}\label{flag1}
 \boxed{\mathfrak{F}_C: F|_C\supset_{a_1} \parK\supset_{a_2} 0\,.}
\end{equation*}
Here we have associated weights $0\leq a_1<a_2<1$ to get a parabolic
vector bundle $(F,\mathfrak{F}_C,a_1,a_2)$. Let us denote $F$ along
with the above parabolic structure by $F_*$.  We call this parabolic
vector bundle the \emph{(dual) parabolic LM bundle}.
\theoremstyle{plain} \newtheorem{parabolicdual1}{Remark}[section]
\begin{parabolicdual1}\label{parabolicdual1}
  Dualizing the exact sequence \eqref{parabonF}, we get the following
  exact sequence on $C$, where $E_{C,A}=F^{\vee}$ is the LM bundle:
 $$0\lr A\lr (E_{C,A})|_C\lr (\parK)^{\vee}\lr 0\,.$$
 This gives a parabolic structure on the LM bundle $E_{C,A}$, by
 associating weights $0\leq a_1<a_2<1$ given by
$$\boxed{(E_{C,A})|_C\supset_{a_1} A\supset_{a_2} 0\,.}$$
We call the above parabolic vector bundle the \emph{parabolic LM
  bundle}. However, in this paper, we work with the \emph{dual
  parabolic LM bundle} $F_*$.
 \end{parabolicdual1}
\subsection{Parabolic Blown up LM bundle}\label{structure2}
Next, we obtain a parabolic structure on the pullback of $F$ to a
blown up surface, along the exceptional divisor. Let $X$, $C$, $A$ and
$F$ be as in the $\mathcal{x}$ \ref{structure1}. Consider a point $q$
of multiplicity one on $C$.

Consider the blow up of $X$ at $q$, say $\tildeX$. Let
$\pi:\tildeX\ra X$ denote the blow down map, and $E\subset\tildeX$,
the exceptional divisor.  Pull back the exact sequence \eqref{LMseq}
by $\pi$ to the blown-up surface $\tildeX$:
\begin{equation*}
  0\ra\pi^*F=:\5tildeF\ra \glH\otimes\mcO_{\tildeX}\ra\pi^*i_*A\ra 0\,.
\end{equation*}
Restricting the above sequence to the exceptional divisor $E$, we get:
\begin{equation}\label{restricttoE}
 0\ra M\ra \5tildeF|_E\simeq F(q)\otimes\mcO_E\ra \glH\otimes\mcO_E\ra A(q)\otimes \mcO_E\ra 0\,.
\end{equation}
Here $M$ is the kernel of the restriction. On $\tildeX$, we have the
following commutative diagram, where $\tildeC$ is the strict transform
of $C$ in $\tildeX$.
\begin{displaymath}
 \xymatrix{E\,\ar@{^{(}->}[r] & \tildeC+E=\pi^*C\ar@{^{(}->}[r] \ar[d]_{\pi_C} & \tildeX\ar[d]^{\pi}\\
 	    & C \ar@{^{(}->}[r]^i & X}
\end{displaymath}
Pulling back the short exact sequence \eqref{evseqonC} on $C$ to
$\pi^*C=\tildeC+E$ and restricting to $E$, we get:
\begin{equation*}\label{curveblonE}
  0\ra A^{\vee}(q)\otimes\mcO_E\ra \glH\otimes\mcO_E\ra A(q)\otimes\mcO_E\ra 0\,.
\end{equation*}
From the above short exact sequence and from \eqref{restricttoE}, we
get:
\begin{equation*}
  0\ra M\ra F(q)\otimes\mcO_E \ra A^{\vee}(q)\otimes\mcO_E\ra 0\,.
\end{equation*}
The kernel $M$ is in fact a line bundle of the form $M=W\otimes\mcO_E$
where $W$ is a one dimensional vector space.  We thus have the
following parabolic structure on $\5tildeF$ along $E$ by associating
weights $0\leq b_1<b_2<1$:
\begin{equation*}\label{partildeF}
  \boxed{\widetilde{\mfk{F}}_E:\5tildeF|_E=F(q)\otimes\mcO_E\supset_{b_1} M=W\otimes{\mcO_E}\supset_{b_2} 0}
\end{equation*}
We denote the corresponding parabolic vector bundle by $\5tildeF_*$
and call it the \emph{(dual) blown up parabolic LM bundle}.
\section{Stability of the parabolic bundles}\label{stability of parabolic bundles}
In this section, we discuss the parabolic stability of the parabolic
vector bundles $F_*$ on $X$ and $\5tildeF_*$ on $\tildeX$.
\subsection{Parabolic stability of $F_*$}
Consider a rank one coherent subsheaf $G$ of $F$ with torsion-free
quotient $F/G$.  Then, $G$ is a line bundle, cf. \cite[Prop. 1.1 \&
1.9]{RH1}.  We give the induced parabolic structure on $G$, induced
from the parabolic vector bundle $F_*$ defined by the parabolic
structure {${ F|_C\supset_{a_1} \parK\supset_{a_2} 0}$.}  By the
procedure explained in Definition \ref{inducedparabolicintro}, the
induced structure is given as follows.
\begin{inducedstructureonG}\label{inducedstructureonG}
  The induced parabolic structure on $G$ along $C$ is given by
 $$G|_C\supset_{a_1} 0\quad\trm{or}\quad G|_C\supset_{a_2} 0\,.$$
\end{inducedstructureonG}
\begin{proof}
  Let $e':G\hra F$ denote the inclusion. Since $F/G$ is torsion-free,
  the restriction $e'|_C :G|_C\hra F|_C$ is also an injection.  Denote
  $e:=e'|_C$. The induced quasi-parabolic structure on $G$ is given as
  under.
\begin{displaymath}
  \xymatrix{0\  \ar@{^{(}->}[r] & \parK\ \ar@{^{(}->}[r] & F|_C\\
    0\ \ar@{^{(}->}[r] & e^{-1}(\parK)=:G'\ \ar@{^{(}->}[r]\ar@{^{(}->}[u] & e^{-1}(F|_C)=G|_C \ar@{^{(}->}[u]}
\end{displaymath}
Since $G'=e^{-1}(\parK)$ is a subsheaf of the line bundle $G|_C$, $G'$
is either of rank 0 or rank 1.
\begin{itemize}
\item If $G'$ is of rank 0, then since $G'\hra \parK$, it is
  torsion-free and hence $G'=0$. In this case, the parabolic structure
  on $G$ is given by $G|_C\supset 0$ with weight $a_1$.
 \item If $G'$ is of rank 1, then the quotient $\frac{G|_C}{G'}$ 
 is of rank 0, and is a subsheaf of the quotient 
 $\frac{F|_C}{\parK}\simeq A^{\vee}$. Hence the quotient is 0 and
 $G'\simeq G|_C$. In this case, the parabolic 
 structure on $G$ is given by $G|_C\supset 0$ with weight $a_2$.
\end{itemize}
Hence the induced parabolic structure on $G$ is given by
$\boxed{G|_C\supset_{a_i} 0}\,$ for $i=1,2$.
\end{proof}
We now prove Theorem \ref{introparstable} regarding the stability of
$F_*$ with respect to an ample line bundle $L$ on $X$.
\begin{proof}[Proof of Theorem \ref{introparstable}]
  From Definition \ref{parabolicstability}, we recall that in order to
  check parabolic stability of $F_*$, it is enough to check the
  inequality of parabolic slopes for rank one subsheaves $G$ of $F$
  with torsion-free quotient equipped with the induced structure. By
  Lemma \ref{inducedstructureonG}, the induced parabolic structure
  $G_*$ is given by $G|_C\supset_{a_i} 0$ for $i=1,2$.

  We know that, $\mu_L(F)=\frac{-C\cdot L}{2}$. Since $F$ is
  $\mu_L$-stable, $\mu_L(G)<\mu_L(F)$. Also
  $\mu_L(G)=c_1(G)\cdot L\in\mathbb{Z}$.  Thereby,
\begin{equation*}
\mu_L(G) \leq \left\{
  \begin{array}{rl}
    \mu_L(F)-1 & \text{if } (C\cdot L)\trm{ is even},\\
    \mu_L(F)-\frac{1}{2} & \text{if } (C\cdot L)\trm{ is odd}.
\end{array} \right.
\end{equation*}
The parabolic weights of $F_*$ and $G_*$ are given by (cf. Definition
\ref{parabolicweight}):
$$\mu\trm{-wt}(F)=(a_1+a_2)(C\cdot L)\trm{ and } \mu\trm{-wt}(G)=a_i(C\cdot L)\trm{ for } i=1, 2\,.$$
Next we compute the parabolic slopes of $F_*$ and $G_*$.
$$\trm{par}\mu_L(F)=\frac{c_1(F)\cdot L +\mu\trm{-wt}(F)}{2}=\mu_L(F)+\frac{(a_1+a_2)(C\cdot L)}{2}\,\trm{ and,}$$
$$\trm{par} \mu_L(G)=\mu_L(G)+a_i(C\cdot L)\trm{ where }i=1\trm{ or }2\,.$$
Since $a_1<a_2$, it is enough to prove that
$\mu_L(G)+a_2(C\cdot L)< \trm{par}\mu_L(F)$ in order to show that
$F_*$ is parabolic $\mu_L$-stable.
\begin{itemize}
\item[(a)] \underline{$C.L$ is even} - In this case, our condition on
  the $a_i$-s, i.e. $a_2-a_1<\frac{2}{(C\cdot L)}$ gives
  $-1< \frac{(a_1-a_2)(C\cdot L)}{2}$.  Thus,
 \begin{eqnarray*}
   \mu_L(G)+a_2(C\cdot L) & \leq & \mu_L(F) -1+ a_2(C\cdot L)\\
                          & < & \mu_L(F)+\frac{(a_1+a_2)(C\cdot L)}{2}=\trm{par}\mu_L(F)\,.
\end{eqnarray*}
\item[(b)] \underline{$C.L$ is odd} - In this case, the condition on
  $a_i$-s is $a_2-a_1<\frac{1}{(C\cdot L)}$. This gives
  $-\frac{1}{2}<\frac{(a_1-a_2)(C\cdot L)}{2}$.  Thus,
 \begin{eqnarray*}
\mu_L(G)+a_2(C\cdot L) & \leq & \mu_L(F) -\frac{1}{2}+ a_2(C\cdot L)\\
& < & \mu_L(F)+\frac{(a_1+a_2)(C\cdot L)}{2}=\trm{par}\mu_L(F)\,. 
\end{eqnarray*}
 \end{itemize}
 By parts (a) and (b), $F_*$ is parabolic $\mu_L$-stable.
\end{proof}
\theoremstyle{plain}
\newtheorem{parabolicK3stable}[inducedstructureonG]{Corollary}
\newtheorem{projectivespparabolic}[inducedstructureonG]{Corollary}
\begin{parabolicK3stable}\label{parabolicK3stable}
  Let $X$ be a smooth, projective K3 surface and $L\in \emph{Pic}(X)$
  an ample line bundle such that a general curve $C\in | L |$ has
  genus $g$, Clifford dimension $1$ and maximal gonality $k $.  Let
  the Brill-Noether number $\rho ( g, 1, d ) > 0$. For a general
  $C\in |L|$ , consider the rank $2$ dual LM bundle $F:=F_{C,A}$
  associated with a general complete, base-point free $g_d^1$, say $A$
  on $C$. Then the dual Parabolic LM bundle $(F,\mfk{F}_C,a_1,a_2)$ is
  parabolic $\mu_L$-stable, whenever
  $$a_2-a_1<\frac{1}{g-1}.$$
\end{parabolicK3stable}
\begin{proof}
  Lelli-Chiesa \cite[Theorem 4.3]{ML} proves that such an $F=F_{C,A}$
  is $\mu_L$-stable. The proof follows from Theorem
  \ref{introparstable} by observing that on a K3 surface
  $C\cdot L=(C^2)=2g-2$, which is even.
\end{proof}
\begin{projectivespparabolic}\label{projectivespparabolic}
  Let $X=\mbb{P}^2_{\mathbb{C}}$, the projective plane. Consider
  $\mcO(d)$ for an \emph{odd} $d>0$. Let $C\in |\mcO(d)|$ be a general
  smooth curve, and $A=\mcO(rd)|_C$ (for $r>0$) on $C$. Let
  $V\subset H^0(C,A)$ be a general two dimensional subspace of global
  sections. Consider the dual LM bundle $F:=F_{C,A,V}$ associated to
  the triple $(C,A,V)$, and the associated parabolic structure on $F$,
  say $(F,\mfk{F}_C,a_1,a_2)$ where $a_2-a_1<\frac{1}{d}$. Then
  $(F,\mfk{F}_C,a_1,a_2)$ is parabolic $\mu_{\mcO(1)}$-stable.
\end{projectivespparabolic}
\begin{proof}
  From \cite[Theorem 1.2, Remark 4.10]{NP2}, we observe that the dual
  LM bundle $F$ is $\mu_{\mcO(1)}$-stable. Since $L=\mcO(1)$, for
  $C\in |\mcO(d)|$, we get $C\cdot L=d$ (which is odd in our
  case). The result thus follows.
\end{proof}
\subsection{Parabolic stability of $\5tildeF_*$}\label{stability-2}
The line bundle $\pi^*L$ is not ample on $\tildeX$ because it is
trivial when restricted to the exceptional divisor.  Consider the line
bundle $\Ln$ on $\tildeX$ which is ample (by \cite[Lemma 3]{Br}) for
$n$ sufficiently large,
\begin{equation*}\label{ampleblup}
  \Ln =n \pi^*L \otimes \mcO_{\tildeX}(-E)\,.
\end{equation*}
We study the parabolic stability of $\5tildeF_*$ with respect to the
ample line bundle $\Ln$. We first note the following about the
$\mu_{\Ln}$-stability of $\5tildeF$ in the following Remark.
\begin{pbstable}\label{pbstable}
  If $F$ is a $\mu_L$-stable vector bundle on $X$, then there exists
  an $n_0\in\mathbb{Z}$ such that for all $n\geq n_0$, the vector
  bundle $\5tildeF=\pi^*F$ is a $\mu_{\Ln}$-stable vector bundle on
  $\tildeX$. Refer \cite[Prop. 2.2 (1)]{Nak} for a proof.
\end{pbstable}
Consider a rank 1 coherent subsheaf $\tildeG$ of $\5tildeF$ with
torsion-free quotient. By similar arguments as before, $\tildeG$ is a
line bundle. We give the induced parabolic structure on $\tildeG$
induced from $\5tildeF_*$ which is defined by
${\5tildeF|_E\supset_{b_1} M=W\otimes\mcO_E\supset_{b_2} 0\,.}$
\newtheorem{inducedstructureonblowupG}[inducedstructureonG]{Lemma}
\begin{inducedstructureonblowupG}
  The induced parabolic structure on $\tildeG$ along $E$ is given
  either by
 $$\tildeG|_E\supset_{b_1} 0\quad\trm{or}\quad \tildeG|_E\supset_{b_2} 0\,.$$
\end{inducedstructureonblowupG}
\begin{proof}
  The proof is exactly on the same lines as the proof of Lemma
  \ref{inducedstructureonG}.
\end{proof}

We now prove Theorem \ref{introparstable2} about the
$\mu_{\Ln}$-stability of $\5tildeF_*$ for $n\geq n_0$, where $n_0$ is
as in Remark \ref{pbstable}).
\begin{proof}[Proof of Theorem \ref{introparstable2}]
  Since the dual LM bundle $F$ is $\mu_L$-stable, by Remark
  \ref{pbstable}, there is an integer $n_0$ such that for all
  $n\geq n_0$, $\5tildeF$ is $\mu_{\Ln}$-stable. Note that,
$$\mu_{\Ln}(\5tildeF)=\frac{c_1(\5tildeF)\cdot \Ln}{2}\simeq \frac{\pi^*\mcO_X(-C)\cdot (n\pi^*L-E)}{2}=\frac{-n(C\cdot L)}{2}\,.$$
If $\tildeG$ is any rank 1 subsheaf of $\5tildeF$ with torsion-free
quotient, then $\mu_{\Ln}(\tildeG)<\mu_{\Ln}(\5tildeF)\,.$ Since
$\mu_{\Ln}(\tildeG)=c_1(\tildeG)\cdot \Ln\in\mathbb{Z}$, we get:
\begin{equation*}\label{parmutildeG}
\mu_{\Ln}(\tildeG) \leq \left\{
\begin{array}{rl}
  \mu_{\Ln}(\5tildeF)-\frac{1}{2} & \text{if } (C\cdot L) \trm{ and } n \trm{ are odd },\\
  \mu_{\Ln}(\5tildeF)-1 & \text{otherwise}.
\end{array} \right.
\end{equation*}
As discussed in the previous theorem, in order to check for parabolic
stability it is enough to check the slope inequality for rank 1
subsheaves $\tildeG\subset\5tildeF$ with torsion-free quotient
equipped with the induced parabolic structure. The induced parabolic
structure on such a $\tildeG$ is
$\tildeG|_E\supset_{b_i} 0\trm{ for }i=1,2\,.$ So the parabolic
weights of $\5tildeF$ and $\tildeG$ are:
\[\mu\trm{-wt}(\5tildeF)= (b_1+b_2)E\cdot(n\pi^*L-E)=-(b_1+b_2)E^2\,.\]
Since $(E^2)=-1$, we get $\mu\trm{-wt}(\5tildeF)=b_1+b_2.$ Similarly,
we get that $\mu\trm{-wt}(\tildeG)= b_iE\cdot\Ln= b_i$.  We now
compute the parabolic slopes,
$$\trm{par}\mu_{\Ln}(\tildeG)=\mu_{\Ln}(\tildeG)+\mu\trm{-wt}(\tildeG)=\mu_{\Ln}(\tildeG)+b_i\,.$$
$$\trm{par}\mu_{\Ln}(\5tildeF)=\mu_{\Ln}(\5tildeF)+\frac{\mu\trm{-wt}(\5tildeF)}{2}=\mu_{\Ln}(\5tildeF)+\frac{b_1+b_2}{2}\,.$$

Since $b_1<b_2$, it is enough to check that
$\mu_{\Ln}(\tildeG)+b_2< \trm{par}\mu_{\Ln}(\5tildeF)$. Since
$0\leq b_1<b_2<1$, we have $b_2-b_1<1$. Again we have two cases.
\begin{itemize}
\item[(a)] \underline{Both $n$ and $(C\cdot L)$ are odd -} Note that,
  $-\frac{1}{2}<\frac{b_1-b_2}{2}$. We now have
 \begin{eqnarray*}
  \mu_{\Ln}(\tildeG)+b_2 & \leq & \mu_{\Ln}(\5tildeF)-\frac{1}{2} +b_2 \\
   & < & \mu_{\Ln}(\5tildeF)+\frac{b_1+b_2}{2} = \trm{par}\mu_{\Ln}(\5tildeF)\,.
 \end{eqnarray*}
\item[(b)] \underline{At least one of $n$ or $(C\cdot L)$ is even -}
  Note that $\frac{b_2-b_1}{2}<b_2-b_1<1$, which gives
  $-1<\frac{b_1-b_2}{2}$.  We have
 \begin{eqnarray*}
   \mu_{\Ln}(\tildeG)+b_2 & \leq & \mu_{\Ln}(\5tildeF)-1 +b_2 \\
                          & < & \mu_{\Ln}(\5tildeF)+\frac{b_1+b_2}{2} = \trm{par}\mu_{\Ln}(\5tildeF)\,.
 \end{eqnarray*}
\end{itemize}
Thus from parts (a) and (b), we have that the (dual) parabolic blown
up LM bundle $\5tildeF_*$ is parabolic $\mu_{\Ln}$-stable for all
weights $b_1$ and $b_2$ and for $n\geq n_0$, where $n_0$ is as in
Remark \ref{pbstable}.
\end{proof}
\newtheorem{blownparabolicK3stable}[inducedstructureonG]{Corollary}
\newtheorem{blownprojectivespparabolic}[inducedstructureonG]{Corollary}
\begin{blownparabolicK3stable}
  Let $X$ be a smooth, projective K3 surface and $L\in \emph{Pic}(X)$
  an ample line bundle such that a general curve $C\in | L |$ has
  genus $g$, Clifford dimension $1$ and maximal gonality $k $.  Let
  the Brill Noether number $\rho ( g, 1, d ) > 0$. For a general
  $C\in |L|$ , consider the rank $2$ dual LM bundle $F:=F_{C,A}$
  associated with a general complete, base-point free $g_d^1$ say $A$
  on $C$. Then, there is an integer $n_0$ such that the dual parabolic
  blown up LM bundle $(\5tildeF,\widetilde{\mfk{F}}_E,b_1,b_2)$ is
  parabolic $\mu_{\Ln}$-stable, for all choices of weights and for all
  $n\geq n_0.$
\end{blownparabolicK3stable}
\begin{proof}
  This result again follows from Lelli-Chiesa \cite[Theorem 4.3]{ML},
  whereby such an $F_{C,A}$ is $\mu_L$-stable.
\end{proof}
\begin{blownprojectivespparabolic}
  Let $X=\mbb{P}^2_{\mathbb{C}}$, the projective plane. Consider
  $\mcO(d)$ for an odd number $d>0$. Let $C\in |\mcO(d)|$ be a general
  smooth curve, and $A=\mcO_X(rd)|_C$ for $r>0$ on $C$. Let
  $V\subset H^0(C,A)$ be a general two dimensional subspace of global
  sections. Consider the dual LM bundle $F:=F_{C,A,V}$ associated to
  the triple $(C,A,V)$, and the associated parabolic structure on
  $\5tildeF$, say $(\5tildeF,\widetilde{\mfk{F}}_E,b_1,b_2)$.  Then
  there is an integer $n_0$ such that
  $(\5tildeF,\widetilde{\mfk{F}}_E,b_1,b_2)$ is parabolic
  $\mu_{\Ln}$-stable, where $L=\mcO(1)$ and for all $n\geq n_0$.
\end{blownprojectivespparabolic}
\begin{proof}
  The result follows from \cite[Theorem 1.2]{NP2} which proves that
  the dual LM bundle $F_{C,A,V}$ is $\mu_{\mcO(1)}$-stable.
\end{proof}
\section{Orbifold LM bundle}\label{Orbifold LM bundle}
Recall that the quasi-parabolic structure on the parabolic vector
bundle $F_*$ is given by
$$F|_C=F_1(F|_C)\supset \parK = F_2(F|_C)\supset 0=F_3(F|_C)\,.$$ 
By \cite[Definition 1.1]{MY}, the quasi-parabolic structure can be
alternatively described by a filtration of the vector bundle $F$
itself of the form
$$F=F_1(F)\supset F_2(F)\supset F_3(F)=F(-C)\,.$$
The elements of the two filtrations are related as follows
(cf. \cite[$\mathcal{x}\,$3]{IB}):
\begin{equation*}
  F_i(F|_C)=\frac{F_i(F)}{F(-C)}\,.
\end{equation*}
We obtain the equivalent formulation of the quasi-parabolic structure
on $F_*$ from the following lemma.
\newtheorem{parabolicstructurealt}{Lemma}[section]
\begin{parabolicstructurealt}\label{parabolicstructurealt}
  The quasi-parabolic structure on the (dual) parabolic LM bundle
  $F_*$ is equivalent to the quasi-parabolic structure given by the
  following filtration of $F$:
  \begin{equation}
    \boxed{\mfk{F}:F\supseteq \glH\otimes\mcO_X(-C)\supseteq F(-C)\,.}\nonumber
  \end{equation}
\end{parabolicstructurealt}
\begin{proof}
  By the discussion above, we just need to determine $F_2(F)$. We also
  have, $\displaystyle F_2(F|_C)=F_2(F)/F(-C)\,.$ Hence $F_2(F)$ has
  to sit in the following short exact sequence:
$$ 0\ra F(-C)\ra F_2(F)\lr i_*(\parK)\ra 0\,.$$
Tensoring the exact sequence \eqref{LMseq} which defines $F$ by
$\mcO_X(-C)$, we get
$$0\lr F(-C)\lr \glH\otimes\mcO_X(-C)\lr i_*(\parK)\lr 0\,.$$
Thus, $F_2(F)=\glH\otimes\mcO_X(-C)$.
\end{proof}
So, the parabolic vector bundle $(F,\mfk{F}_C,a_1,a_2)$ from
$\mathcal{x}\,$\ref{structure1} can alternatively be described as
$(F,\mfk{F},a_1,a_2)$ where the flag $\mfk{F}$ is described in Lemma
\ref{parabolicstructurealt}.
\newtheorem{anotherwayparabolic}[parabolicstructurealt]{Remark}
\begin{anotherwayparabolic}\label{anotherwayparabolic}
  Set $a_0=a_2-1$ and $a_3=1$, where $F_*=(F,\mfk{F},a_1,a_2)$ is the
  (dual) parabolic LM bundle.  Let $[t]$ denote the integral part of
  any $t\in\mathbb{R}$.  For each $t\in\mbb{R}$, consider the
  following locally free sheaves $E_t=F_i(F)(-[t]C)$, where
  $a_{i-1}<t-[t]\leq a_i$. Note that the locally free sheaves
  $\{E_t\}$ give a filtration which is decreasing, i.e. if $t\geq t'$,
  then $E_t\subseteq E_{t'}$. This filtration is left continuous.
  Further, this filtration has a jump at some $t$, i.e. for all
  $\epsilon>0$, $E_{t+\epsilon}\neq E_t$ if and only if $t-[t]=a_i$
  where $i=1,2$. Finally, for all $t\in\mbb{R}$, we have
  $E_{t+1}=E_t(-C)$.  This filtration $\{E_t\}_{t\in\mbb{R}}$
  completely determines the parabolic vector bundle
  $(F,\mfk{F},a_1,a_2)$, cf. \cite{MY} and \cite{IB}.  In particular,
  the parabolic vector bundle $F_*$ is completely determined by the
  filtration of vector bundles
\begin{eqnarray*}
 E_t & = & F \trm{ for }t\in[0,a_1]\\
  & = & \glH\otimes\mcO_X(-C) \trm{ for } t\in (a_1,a_2]\\
  & = & F(-C) \trm{ for } t\in(a_2,1].
\end{eqnarray*}
  \end{anotherwayparabolic}
  We recall the Kawamata covering Theorem
  \cite[Prop. 4.1.12]{Laz}. Given a positive integer $N>0$, there is a
  non-singular connected projective variety $Y$ and a finite Galois
  morphism $p:Y\ra X$ with Galois group
  $\Gamma = \trm{Gal}(\trm{Rat}(Y)/\trm{Rat}(X))$ such that $p^*C$ is
  a non-reduced divisor of the form $NC'$ with
  $(p^*C)_{\text{red}}=C'$. Here $C'$ is a smooth curve on $Y$. We
  then have the following commutative diagram, for any $m$ such that
  $1\leq m < N$.
\begin{displaymath}
 \xymatrix{C' \ar@{^{(}->}[r]^{\hspace{-0.3 cm} g}\ar[drr]_{\psi} & mC' \ar@{^{(}->}[r]^{\ f} \ar[dr]^{\phi}  & NC' \ar@{^{(}->}[r]^{j}\ar[d]^{p} & Y \ar[d]^p\\ 
 & & C\ar@{^{(}->}[r]_{i} & X}
\end{displaymath}
Pulling back the exact sequence \eqref{evseqonC} to $mC'$ by $\phi$,
we get:
$$0\ra\phi^*A^{\vee}\ra H^0(C,A)\otimes\mcO_{mC'}\ra
  \phi^*A\ra 0\,.$$ Denote $A'=\phi^*A.$ Then $(A',H^0(C,A))$ is a
base-point free linear system on the curve $mC'$ 
(which is non-reduced when $m>1$). We can thus consider the dual
LM bundle $F'$ on $Y$ associated to the triple $(mC',A',H^0(C,A))$ 
given by the following short exact sequence: 
\begin{equation}\label{LMsequenceKawamata}
 0\ra F'\ra H^0(C,A)\otimes\mcO_Y\ra j_*f_*A'\ra 0\,.
\end{equation}
\newtheorem{LMonKawamataorbifold}[parabolicstructurealt]{Remark}
\begin{LMonKawamataorbifold}
  The dual LM bundle $F'$ on $Y$ associated to the triple
  $(mC',A',H^0(C,A))$ has a natural orbifold structure. Indeed,
  $H^0(C,A)\,\otimes\,\mcO_Y\ra j_*f_*A'$ is a morphism between
  orbifold sheaves on $Y$ compatible with the action of $\Gamma$ on
  both the sheaves. We call this orbifold bundle $F'$ on $Y$ ``the
  orbifold LM bundle''.
\end{LMonKawamataorbifold}
By the construction of Biswas \cite[$\mathcal{x}\,$2c]{IB}, this
orbifold LM bundle corresponds to a uniquely determined parabolic
bundle on $X$. We now determine this parabolic vector bundle.  For any
$t\in \mbb{R}$, define, as in \cite{IB},
\begin{equation}\label{paraboliccontfiltration}
E_t=\Big(p_*\Big(F'\otimes \mcO_Y([-t\cdot N]C')\Big)\Big)^{\Gamma}\,. 
\end{equation}
Just as in Remark \ref{anotherwayparabolic}, this defines a parabolic
sheaf on $ X$, cf. \cite[$\mathcal{x}\,$2]{IB} and
\cite[$\mathcal{x}\,$1]{MY}.  We first have the following preliminary
lemma.  \newtheorem{conormbundledescent}[parabolicstructurealt]{Lemma}
\begin{conormbundledescent}\label{conormalbyndle}
  Let $A$ be a line bundle on $C$. For any $m$ such that $1\leq m<N$,
  consider the line bundle $L':=\psi^*A\otimes\mcO_Y(-mC')|_{C'}$ on
  $C'$. Then, with notations as in the commutative diagram below, the
  invariant sheaf
  $(i_*\psi_*L')^{\Gamma}=(p_*j_*f'_*h'_*g_*L')^{\Gamma}=0$.
\end{conormbundledescent}
 \begin{displaymath}
  \xymatrix{C'\ar@{^{(}->}[r]^{g}\ar[drrr]_{\psi} & mC'\ar@{^{(}->}[r]^{h'}\ar[drr]^{\phi} & (m+1)C'\ar@{^{(}->}[r]^{f'}\ar[dr]^{\chi} & NC'\ar@{^{(}->}[r]^j\ar[d]^p & Y\ar[d]^p\\
			      &			   &			    & C\ar@{^{(}->}[r]_i & X}
 \end{displaymath}
 \begin{proof}
   Corresponding to the closed immersion $h':mC'\hra (m+1)C'$, for
   $m=1,2,\cdots N-1$, we have the surjective morphism of sheaves
   $\mcO_{(m+1)C'}\twoheadrightarrow h'_*\mcO_{mC'}$. The ideal sheaf
   of this closed subscheme is $h'_*g_*\mcO_Y(-mC')|_{C'}$.  This
   gives,
 $$0\ra h'_*g_*\mcO_Y(-mC')|_{C'}\ra \mcO_{(m+1)C'}\ra h'_*\mcO_{mC'}\ra 0\,. $$
 Tensor by $\chi^*A$ (and use projection formula) to get
 $$0\ra h'_*g_*(\psi^*A\otimes\mcO_Y(-mC')|_{C'})\simeq h'_*g_*L'\ra \chi^*A\ra h'_*(\phi^*A)\ra 0\,. $$
 Pushforward this sequence to $X$ by applying $p_*j_*f'_*$:
 $$0\ra p_*j_*f'_*h'_*g_*L'\ra p_*j_*f'_*\chi^*A\ra p_*j_*f'_*h'_*\phi^*A\ra 0\,. $$
 That is, 
  $$ 0\lr i_*\psi_*L'\lr i_*\chi_*\chi^*A\lr i_*\phi_*\phi^*A\lr 0\,.$$
  Note that the above is a sequence of pushforwards to $X$ by $p$ of
  $\Gamma$-sheaves from $Y$. Therefore, we can take
  $\Gamma$-invariants of each of these sheaves. The sequence of
  $\Gamma$-invariant sheaves continues to be exact,
  cf. \cite[$\mathcal{x}$1.4.3]{AK}:
 $$0\lr (i_*\psi_*L')^{\Gamma}\lr (i_*\chi_*\chi^*A)^{\Gamma}\lr (i_*\phi_*\phi^*A)^{\Gamma}\lr 0\,.$$
 However, $(i_*\chi_*\chi^*A)^{\Gamma}\simeq i_*A\simeq (i_*\phi_*\phi^*A)^{\Gamma}$. 
 Thus, $(i_*\psi_*L')^{\Gamma}=0$.
\end{proof}
We now prove Theorem \ref{intro-parabolic-orbifold} stated in the
introduction.
\begin{proof}[Proof of Theorem \ref{intro-parabolic-orbifold}]
  We need to prove that the orbifold bundle on $Y$ corresponding to
  the dual parabolic LM bundle $(F,\mfk{F}_C, 0,\frac{N-m}{N})$ on $X$
  is the dual LM bundle $F'$. We do that by showing that the parabolic
  bundle on $X$ corresponding to $F'$ is the required one.

  We recall the commutative diagram where $1\leq m< N$:
\begin{displaymath}
  \xymatrix{C' \ar@{^{(}->}[r]^{\hspace{-0.3 cm} g}\ar[drr]_{\psi} & mC' \ar@{^{(}->}[r]^{\ f} \ar[dr]^{\phi}  & NC' \ar@{^{(}->}[r]^{j}\ar[d]^{p} & Y \ar[d]^p\\ 
    & & C\ar@{^{(}->}[r]_{i} & X}
\end{displaymath}
For all $a=1,2,\ldots,N$, we have the following exact sequence which
we will use repeatedly in the proof:
\begin{equation}\label{shortexwitha - 1}
  0\ra F'(-aC')\ra F'(-(a-1)C')\ra j_*f_*g_*\big(F'\big(-(a-1)C'\big)\big|_{C'}\big)\ra 0\,.
\end{equation}
Note that, just as in the short exact sequence \eqref{parabonF},
$F'|_{mC'}$ fits in the following short exact sequence of sheaves on
$mC'$:
\begin{equation*}
  0\lr \phi^*A\otimes\mcO_Y(-mC')|_{mC'}\lr F'|_{mC'}\lr \phi^*A^{\vee}\lr 0\,.
\end{equation*}
We restrict this to $C'$:
\begin{equation}\label{restricttoC'}
  0\lr \psi^*A\otimes\mcO_Y(-mC')|_{C'}\lr F'|_{C'}\lr \psi^*A^{\vee}\lr 0\,. 
\end{equation}
By equation \eqref{paraboliccontfiltration}, we have
$E_0=\Big(p_*F'\Big)^{\Gamma}$.  We further have that, if
$t\in (\frac{a-1}{N},\frac{a}{N}]$ where $a=1,2,\cdots, N-1, N$, then,
 $$E_t=\Big(p_*\Big(F'\otimes \mcO_Y(-aC')\Big)\Big)^{\Gamma}\,. $$
 We now compute $E_t$ as we vary $t\in [0,1]$, and this filtration of
 sheaves will completely determine the parabolic bundle on $X$ corresponding to $F'$.\\
 \\ \noindent \textbf{1.}\quad \underline{\textbf{$t=0$}}\\
 Pushforward the exact sequence \eqref{LMsequenceKawamata} to $X$ and
 consider the sequence of $\Gamma$-invariant sheaves (Taking
 $\Gamma$-invariants is an exact functor).  Since
 $p_*j_*f_*A'\simeq i_*\phi_*\phi^*A$, we get
 $$0\ra (p_*F')^{\Gamma}\ra H^0(C,A)\otimes
   (p_*\mcO_Y)^{\Gamma}\ra (i_*\phi_*\phi^*A)^{\Gamma}\ra 0\,.$$
 Thus we get,
 $$0\ra (p_*F')^{\Gamma}\ra H^0(C,A)\otimes \mcO_X\ra
 i_*A\ra 0\,.$$ Therefore $(p_*F')^{\Gamma}=E_0=F\,.$\\
 \\ \noindent \textbf{2.}\quad \underline{\textbf{$t\in \big(0,\frac{N-m}{N}\big]$}}\\
 Consider $t\in\big(\frac{a-1}{N},\frac{a}{N}\big]$ for
 $a=1,2,\ldots,N-m$. Then,
 $E_t=\Big(p_*\Big(F'\otimes \mcO_Y(-aC')\Big)\Big)^{\Gamma}\,. $ We
 consider the pushforward of the sequence \eqref{shortexwitha - 1} to
 $X$ by $p$, noting that
 $p_*j_*f_*g_*\big(F'\big(-(a-1)C'\big)\big|_{C'}\big)\simeq
 i_*\psi_*\big(F'\big(-(a-1)C'\big)\big|_{C'}\big)$, to get:
 $$ 0\ra p_*F'(-aC')\ra p_*F'(-(a-1)C')\ra i_*\psi_*F'\big(-(a-1)C'\big)\big|_{C'}\ra 0\, .$$
 Now the sheaf $i_*\psi_*\big(F'\big(-(a-1)C'\big)\big|_{C'}\big)$
 fits in the following short exact sequence that we get from
 \eqref{restricttoC'}:
\begin{multline}\label{exactwithGamma}
  0\ra
  i_*\psi_*\big(\psi^*A\otimes\mcO_Y\big((-(a-1)-m)C'\big)\big|_{C'}\big)\ra
  i_*\psi_*F'\big(-(a-1)C'\big)\big|_{C'} \\ \ra
  i_*\psi_*\big(\psi^*A^{\vee}\otimes
  \mcO_Y\big(-(a-1)C'\big)\big|_{C'}\big)\ra 0\,.
\end{multline}
For $a = 2,\ldots, N-m$, we have
$\big(i_*\psi_*\big(\psi^*A\otimes\mcO_Y\big((-(a-1)-m)C'\big)\big|_{C'}\big)\big)^{\Gamma}=0$
and
$\big(i_*\psi_*\big(\psi^*A^{\vee}\otimes
\mcO_Y\big(-(a-1)C'\big)\big|_{C'}\big)\big)^{\Gamma}=0$ by Lemma
\ref{conormalbyndle}. Hence,
$(i_*\psi_*F'(-(a-1)C')|_{C'})^{\Gamma}=0$ and thus
$(p_*F'(-aC'))^{\Gamma}=(p_*F'(-(a-1)C'))^{\Gamma}$ for
$a= 2,3,\cdots, N-m$.

  Now it is sufficient to compute $(p_*F'(-C'))^{\Gamma}$ which is in
  fact $E_t$ for $t \in (0, \frac{1}{N}]$ i.e. the case $a = 1$
  above. 
  Set $a = 1$ in \eqref{shortexwitha - 1} and pushforward
  the sequence of $\Gamma$-sheaves on $Y$ to $X$:
 \begin{equation}\label{weight1/n}
 0\ra p_*F'(-C')\ra p_*F'\ra p_*j_*f_*g_*(F'|_{C'})\simeq i_*\psi_*(F'|_{C'})\ra 0. 
 \end{equation}
 Next, setting $a=1$ in the short exact sequence
 \eqref{exactwithGamma}, we get:
\begin{equation}
 0\ra i_*\psi_*(\psi^*A\otimes\mcO_Y(-mC')|_{C'})\ra i_*\psi_*F'|_{C'} \ra i_*\psi_*\psi^*A^{\vee}\ra 0\,.
\end{equation}
By Lemma \ref{conormalbyndle},
$(i_*\psi_*(\psi^*A\otimes\mcO_Y(-mC')|_{C'}))^{\Gamma}=0$ since
$m < N$. Also, we have
$(i_*\psi_*\psi^*A^{\vee})^{\Gamma}=i_*A^{\vee}$. Hence,
$(i_*\psi_*F'|_{C'})^{\Gamma}\simeq i_*A^{\vee}.$ Considering the
invariants of the short exact sequence \eqref{weight1/n}, we get
$$ 0\ra (p_*F'(-C'))^{\Gamma}\ra F\ra i_*A^{\vee}\ra 0. $$
From the short exact sequence \eqref{dualLMseq}, it follows that
$(p_*F'(-C'))^{\Gamma}\simeq H^0(C,A)\otimes\mcO_X(-C)$. Thus
$E_t=H^0(C,A)\otimes\mcO_X(-C)\trm{ for }t\in(0,1/N]$ and $(p_*F'(-aC'))^{\Gamma} =(p_*F'(-C'))^{\Gamma} = H^0(C,A)\otimes\mcO_X(-C)$ for $a = 2,\ldots,N-m$. Thereby,
$$E_t=H^0(C,A)\otimes\mcO_X(-C)\quad\text{for all}\quad t\in \big(0,\tfrac{N-m}{N}\big]\,.$$

\noindent \textbf{3.}\quad \underline{\textbf{$t\in \big(\frac{N-m}{N},\frac{N-m+1}{N}\big]$}}\\
In this case, we have
$E_t=\Big(p_*\Big(F'\otimes
\mcO_Y(-(N-m+1)C')\Big)\Big)^{\Gamma}$. Since
$p_*j_*f_*g_*(F'(-(N-m)C')|_{C'})=i_*\psi_*(F'(-(N-m)C')|_{C'})$, from
\eqref{shortexwitha - 1} we have
\begin{multline}\label{1weightN/N}
 0\ra p_*F'(-(N-m+1)C')\ra p_*F'(-(N-m)C') \\ \ra i_*\psi_*(F'(-(N-m)C')|_{C'})\ra 0\,. 
\end{multline}
From \eqref{restricttoC'}, we get
 $$0\ra (i_*\psi_*(\psi^*A\otimes\mcO_Y(-NC')|_{C'}))^{\Gamma}\ra (i_*\psi_*F'(-(N-m)C')|_{C'})^{\Gamma}$$
 $$\ra (i_*\psi_*(\psi^*A^{\vee}\otimes \mcO_Y(-(N-m)C')|_{C'}))^{\Gamma}\ra 0\,.$$
 As before, the right hand side term of the above sequence is
 zero. But note that
 $(i_*\psi_*(\psi^*A\otimes\mcO_Y(-NC')|_{C'}))^{\Gamma}=
 (i_*\psi_*\psi^*(A\otimes\mcO_X(-C)|_{C}))^{\Gamma}\simeq
 i_*(A\otimes\mcO_X(-C)|_{C})$ and hence
 $(i_*\psi_*F'(-(N-m)C')|_{C'})^{\Gamma}\simeq
 i_*(A\otimes\mcO_X(-C)|_{C})$. Finally, taking $\Gamma$-invariants of
 \eqref{1weightN/N}, we get:
\begin{equation*}\label{weightN/N-1}
  0\ra (p_*F'(-(N-m+1)C'))^{\Gamma}\ra H^0(C,A)\otimes\mcO_X(-C)\ra i_*(A\otimes\mcO_X(-C)|_{C})\ra 0\,. 
\end{equation*}
From the exact sequence \eqref{LMseq}, we get
$(p_*F'(-(N-m+1)C'))^{\Gamma}\simeq F\otimes\mcO_X(-C)$.  Thus for
$t\in\big(\frac{N-m}{N},\frac{N-m+1}{N}\big]$, we get
$E_t=F\otimes\mcO_X(-C) = F(-C)$. Since this is the expected end of
the flag for the parabolic structure, we deduce, based on the known
properties of $E_t$ from Biswas \cite{IB}, that $E_t = F(-C)$ for
$t \in \big(\frac{N-m}{N}, 1\big]$.
 
\noindent  Thus, we have the following sheaves $E_t$ for $t\in [0,1]$: 
\begin{enumerate}
 \item $E_0= F$,
 \item $E_t=H^0(C,A)\otimes\mcO_X(-C)$ for $t\in\big(0,\frac{N-m}{N}\big]$,
 \item $E_t=F(-C)$ for $t\in \big(\frac{N-m}{N},1\big]$.
\end{enumerate}
As there is a jump at $t=0$ and at $t=\frac{N-m}{N}$, these give the
parabolic weights. Indeed, the parabolic vector bundle on $X$
corresponding to $F'$ is:
$$\boxed{F\supset_{0} H^0(C,A)\otimes\mcO_X(-C)\supset_{\frac{N-m}{N}} F(-C)\,}$$
or equivalently the $(F,\mfk{F}_C,0,\frac{N-m}{N})$.  This is
precisely the description of the required parabolic sheaf.
\end{proof}
The above theorem shows that the LM bundles behave well under the
parabolic-orbifold bundle correspondence.
\section{Proof of Theorem \ref{intronewsemistLMbundles}}\label{semist-appl-section}
We keep notations as in previous sections. Let $L$ be an ample line
bundle on $X$.  \newtheorem{theorem}{Corollary}[section]
\begin{theorem}\label{semist-appl}
  The dual parabolic LM bundle $(F, \mfk{F}_C,a_1,a_2)$ with $a_1=0$
  and $a_2=\frac{N-m}{N}$ on $X$ is parabolic $\mu_L$-semistable if
  and only if the corresponding Orbifold LM bundle $F'$ associated to
  the triple $(mC',\phi^*A,H^0(C,A))$ on $Y$ is
  $\mu_{p^*L}$-semistable.
\end{theorem}
\begin{proof} This is a direct consequence of \cite[Lemma 3.13]{IB}
  and the fact that an orbifold bundle is orbifold semistable with
  respect to $p^*L$ if and only if it is $\mu_{p^*L}$-semistable on
  $Y$, cf. \cite[Lemma 2.7]{IB}.
\end{proof}
Consequently, we prove Theorem \ref{intronewsemistLMbundles}.
\begin{proof}[Proof of Theorem \ref{intronewsemistLMbundles}]
  Keep notations as in the statement of the Theorem.
 \begin{enumerate}
 \setlength\itemsep{0.3 cm}
\item[(a)] In this case, $X=\mbb{P}^2$ and $C\in |\mcO(d)|$ is a
  general smooth degree $d$ curve, where $d$ is odd. Let
  $A=\mcO(rd)|_C$ for any $r\geq 1$, and $V\subset H^0(C,A)$ be a
  general two dimensional subspace. Let $F$ be the dual LM bundle on
  $X$ corresponding to $(C,A,V)$ which is $\mu_{\mcO(1)}$-stable,
  cf. \cite[Theorem 1.2]{NP2}.  Given that $m$ is an integer such that
  $\frac{N-m}{N}<\frac{1}{d}$.  Consider the dual parabolic LM bundle
  $F_*=(F,\mfk{F}_C,a_1,a_2)$ where the weights $a_1=0$ and
  $a_2=\frac{N-m}{N}<\frac{1}{d}$.  By Corollary
  \ref{projectivespparabolic}, $F_*$ is parabolic
  $\mu_{\mcO(1)}$-stable.  By Corollary \ref{semist-appl}, the dual LM
  bundle on $Y$ corresponding to the triple $(mC',A', V)$ on $Y$ is
  $\mu_{p^*\mcO(1)}$-semistable.
\item[(b)] Here $X$ is a smooth projective K3 surface and $L$ is an
  ample line bundle on $X$ such that a general curve $C\in | L |$ has
  genus $g$, Clifford dimension $1$ and maximal gonality $k $.  We
  have $\rho ( g, 1, d ) > 0$. For a general smooth $C\in |L|$ ,
  consider the rank $2$ dual LM bundle $F:=F_{C,A}$ associated with a
  general complete, base-point free $g_d^1$, say $A$ on $C$.  By
  \cite{ML}, $F$ is $\mu_L$-stable. We have an integer $m$ such that
  $\frac{N-m}{N}<\frac{1}{g-1}$.  Then, by Corollary
  \ref{parabolicK3stable}, the dual parabolic LM bundle
  $(F,\mfk{F}_C,a_1,a_2)$ with weights $a_1=0$ and
  $a_2=\frac{N-m}{N}<\frac{1}{g-1}$ is parabolic $\mu_L$-stable.  Then
  by Corollary \ref{semist-appl}, the LM bundle on $Y$ corresponding
  to the triple $(mC',A',H^0(C,A))$ is $\mu_{p^*L}$-semistable.
\item[(a')] As a particular case of (a), let $C\in |\mcO(1)|$ be a
  line, $A$ is a globally generated line bundle on $C$ and $V$ be as
  earlier.  Again, the dual LM bundle $F$ corresponding to $(C,A,V)$
  is $\mu_{\mcO(1)}$-stable. By Corollary \ref{projectivespparabolic},
  the dual parabolic LM bundle $(F,\mfk{F}_C,a_1,a_2)$ is parabolic
  $\mu_{\mcO(1)}$-stable for any weights $a_i$ such that $a_2-a_1<1$.
  Set $a_1=0$ and $a_2=\frac{N-1}{N}$, i.e. $m=1$ in part (a).
  
  By part (a) the dual LM bundle $F'$ on $Y$ corresponding to the
  triple $(C', A', V)$ is $\mu_{p^*\mcO(1)}$-semistable on $Y$.  Let
  $l=\trm{deg}\,A'$. There is a flat family of LM bundles on $Y$
  parametrized by an open subset of the elements of the Brill-Noether
  variety $\mathcal{G}^1_l(|\mcO_Y(C')|)$, cf.
  \cite[$\mathcal{x}\,3$]{NP1}. Since semistability of vector bundles
  is an open condition in flat families \cite[Prop. 2.3.1]{HL}, there
  is an irreducible component of $\mathcal{G}^1_l(|\mcO_Y(C')|)$
  corresponding to $\mu_{p^*\mcO(1)}$- semistable LM bundles.
 
\item[(b')] We consider a particular case of (b), where $L$
  additionally satisfies $L^2=2$. Then a general curve $C\in | L |$
  has genus 2, Clifford dimension $1$ and maximal gonality $k $.
  Suppose $A$ is a complete base-point free $g^1_d$ on $C$ where $d$
  is a positive integer such that the Brill-Noether number
  $\rho(2,1,d)>0$ (note that $d=3$ satisfies the requirement that
  $\rho(2,1,d)>0$ and completeness). Then the dual LM bundle
  $F:=F_{C,A}$ associated to a general $C$ and $A$ is $\mu_L$-stable.
  By Corollary \ref{parabolicK3stable}, the dual parabolic bundle
  $(F,\mfk{F}_C,a_1,a_2)$ is parabolic $\mu_{L}$-stable for any
  weights $a_i$ such that $a_2-a_1<\frac{1}{g-1}=1$.
  
  Set weights $a_1=0$ and $a_2=\frac{N-1}{N}$, i.e. $m=1$ in case (b).
  Then $(F,\mfk{F}_C,0,\frac{N-1}{N})$ is parabolic $\mu_L$-stable.
  Just as in part (b), the dual LM bundle on $Y$ corresponding to the
  triple $(C',A',H^0(C,A))$ is $\mu_{p^*L}$-semistable.  Consider the
  flat family of LM bundles parametrized by the elements of an open
  subset of the Brill-Noether variety
  $\mathcal{G}^1_{l}(|\mcO_Y(C')|)$, cf. \cite[$\mathcal{x}\,3$]{NP1}.
  Here $l=\trm{deg}\, A'$.  Since semistability of vector bundles is
  an open condition in flat families \cite[Prop. 2.3.1]{HL}, there is
  an irreducible component of $\mathcal{G}^1_{l}(|\mcO_Y(C')|)$
  corresponding to $\mu_{p^*L}$-semistable LM bundles.
 \end{enumerate}
\vspace{-20pt}
\end{proof}
\newtheorem{egK3surface}[theorem]{Remark}
\begin{egK3surface}
  Consider $\sigma:X\ra \mbb{P}^2$, a K3 surface which is a double
  cover of the projective plane branched along a smooth sextic curve,
  and $L=\sigma^*\mcO_{\mbb{P}^2}(1)$. Then $X$ and $L$ satisfy the
  requirements of part (b') of the above theorem. That is, $(L^2)=2$
  and a general curve $C\in |L|$ has genus 2, Clifford dimension 1 and
  constant gonality 2, cf. \cite[Theorem A and Prop. 3.3]{CP}.
\end{egK3surface}

\end{document}